\documentclass[11pt]{article}

\usepackage{amsmath,amssymb,amsthm,color,graphicx}
\usepackage{euscript}
\usepackage[utf8]{inputenc}
\usepackage[colorlinks=true]{hyperref}

\usepackage[top=1.25in,bottom=1.25in,left=1.25in,right=1.25in]{geometry}

\DeclareMathOperator{\conv}{conv}
\DeclareMathOperator{\relint}{relint}
\newcommand{\R}{\mathbb R}
\def\F{\mathcal{F}}
\def\ZZ{{\mathbb{Z}}}

\def\Q{{\mathbb{Q}}}
\def\G{{\mathcal G}}
\def\T{{\mathcal{T}}}
\def\bd{{\sf bd}}
\def\reals{{\mathbb R}}
\def\CH{{{\EuScript{CH}}}}
\def\A{{\EuScript{A}}}
\def\B{{\EuScript{B}}}
\def\I{{\mathcal I}}
\def\H{{{\cal{G}}}}
\def\V{{{\EuScript{V}}}}
\def\E{{{\mathcal{E}}}}

\newtheorem{theorem}{Theorem}
\newtheorem{conjecture}{Conjecture}
\newtheorem{claim}[theorem]{Claim}
\newtheorem*{problem}{Problem} 
\newtheorem{lemma}[theorem]{Lemma}

\theoremstyle{definition}
\newtheorem*{definition}{Definition}
\newtheorem{remark}{Remark}

\begin{document}
\title{Further Consequences of the Colorful Helly Hypothesis\footnote{The project leading to this application has received funding from European Research Council (ERC)
under the European Union's Horizon 2020 research and innovation programme under grant agreement No. 678765. The first and third authors were also supported
grant 1452/15 from Israel Science Foundation. The third author was also supported by Ralph Selig Career Development Chair in Information Theory and grant 2014384 from the U.S.-Israeli Binational Science Foundation. The second author was supported by PAPIIT project IA102118}}
\author{Leonardo Martínez-Sandoval\thanks{Department of Computer Science, Ben-Gurion University of the Negev, Beer-Sheva, Israel 84105
{\tt leontz@im.unam.mx}. Also supported
by grant 1452/15 from Israel Science Foundation.} \and Edgardo Roldán-Pensado\thanks{Centro de Ciencias Matemáticas, UNAM campus Morelia, Michoacán, México, 58190, {\tt e.roldan@im.unam.mx}.} \and Natan Rubin\thanks{Department of Computer Science, Ben-Gurion University of the Negev, Beer-Sheva, Israel 84105,
{\tt rubinnat.ac@gmail.com}. Also supported
by grant 1452/15 from Israel Science Foundation and by grant 2014384 from the U.S.-Israeli Binational Science Foundation.}}
\maketitle

\begin{abstract}
    Let $\F$ be a family of convex sets in $\reals^d$, which are colored with $d+1$ colors.
    We say that $\F$ satisfies the Colorful Helly Property if every rainbow selection of $d+1$ sets, one set from each color class,
    has a non-empty common intersection.
    The Colorful Helly Theorem of Lovász states that for any such colorful family $\F$ there is a color class $\F_i\subset \F$, for $1\leq i\leq d+1$, whose sets have a non-empty intersection. 
    We establish further consequences of the Colorful Helly hypothesis.
    In particular, we show that for each dimension $d\geq 2$ there exist numbers $f(d)$ and $g(d)$ with the following property: either one can find an additional color class whose sets can be pierced by $f(d)$ points, or all the sets in $\F$ can be crossed by $g(d)$ lines.  
\end{abstract}

\maketitle

\noindent\textbf{Keywords:} \textit{geometric transversals, convex sets, colorful Helly-type theorems, line transversals, weak epsilon-nets, transversal numbers\\}

\section{Introduction}
\subsection{Helly-type theorems}
Let $\F$ be a finite family of convex sets in $\reals^d$. We say that a collection $X$ of geometric objects (e.g., points, lines, or $k$-flats -- $k$-dimensional affine subspaces of $\reals^d$) is a \emph{transversal} to $\F$, or that $\F$ can be \emph{pierced} or \emph{crossed} by $X$, if each set of $\F$ is intersected by some member of $X$. For an integer $j$ we use the symbol $\binom{\mathcal{F}}{j}$ to denote the collection of subfamilies of $\mathcal{F}$ of size $j$.

The 1913 theorem of Helly \cite{Helly1923} states that a finite family $\F$ of convex sets has a non-empty intersection (i.e., $\F$ can be pierced by a single point) if and only if each of its subsets $\F'\subset \F$ of size at most $d+1$ can be pierced by a point.

In the past 50 years Geometric Transversal Theory has been preoccupied with the following questions (see e.g. \cite{Amenta2017}, \cite{Danzer1963}, \cite{Eckhoff1990}, \cite{Goodman1993},  \cite{Wenger2017}):
\begin{itemize}
    \item Does Helly's Theorem generalize to transversals by $k$-flats, for $1\leq k\leq d-1$?
    \item Given that a significant fraction of the $(d+1)$-tuples $\F'\in \binom{\F}{d+1}$ have a non-empty intersection, can $\F$, or at least some fixed fraction of its members, be pierced by constantly many points?
\end{itemize}

The first question has been settled to the negative already for $k=1$. For instance, Santaló \cite{Santalo1940} and Danzer \cite{Danzer1957} observed that for any $n\geq 3$ there are families $\F$ of $n$ convex sets in $\reals^2$ so that any $n-1$ of the sets can be crossed by a single line transversal while no such transversal exists for $\F$.
Nevertheless, Alon and Kalai \cite{Alon1995} show that the following almost-Helly property holds for $k=d-1$: If every $d+1$ (or fewer) of the sets of $\F$ can be crossed by a hyperplane, then $\F$ admits a transversal by $h$ hyperplanes, where the number $h=h(d)$ depends only on the dimension $d$.

While the properties of hyperplane transversals largely resemble those of point transversals, this is not the case for transversals by $k$-flats of intermediate dimensions $1\leq k\leq d-2$. For example, Alon et al. \cite{Alon2002} showed that for every integers $d\geq 3,m$ and $n_0\geq m+4$ there is a family of at least $n_0$ convex sets so that any $m$ of the sets can be crossed by a line but no $m+4$ of them can;
this phenomenon can be largely attributed to the complex topological structure of the space of transversal $k$-flats.

The second question gave rise to a plethora of inter-related results in discrete geometry and topological combinatorics.

\begin{theorem}[Fractional Helly's Theorem] \label{theorem:Fractional}
    For any $d\geq 1$ and $\alpha>0$ there is a number $\beta=\beta(\alpha,d)>0$ with the following property: For every finite family $\F$ of convex sets in $\reals^d$ so that at least $\alpha\binom{|\F|}{d+1}$ of the $(d+1)$-subsets $\F'\in \binom{\F}{d+1}$ have non-empty intersection, there is a point which pierces at least $\beta |\F|$ of the sets of $\F$.
\end{theorem}

Theorem \ref{theorem:Fractional} was proved by Liu and Katchalski in \cite{Katchalski1979} and it is one of the key ingredients in the proof of the so called Hadwiger-Debrunner $(p,q)$-Conjecture \cite{Hadwiger1957} by Alon and Kleitman \cite{Alon1992}.

\begin{definition}
    We say that a family of convex sets has the {\it $(p,q)$-property}, for $p\geq q$, if for any $p$-subset $\F'\in \binom{\F}{p}$ there is a $q$-subset $\F''\in\binom{\F'}{q}$ with non-empty common intersection $\bigcap \F''\neq \emptyset$.
\end{definition}

\begin{theorem}[The $(p,q)$-theorem \cite{Alon1992}] \label{theorem:PQ}
    For any $d\geq 1$ and $p\geq q\geq d+1$ there is a number $P=P(p,q,d)$ with the following property: Any finite family $\F$ of convex sets in $\reals^d$ with the $(p,q)$-property can be pierced by $P$ points.
\end{theorem}

The proof of Theorem \ref{theorem:PQ} combines Theorem \ref{theorem:Fractional} with the following result of independent interest.

\begin{theorem}[Weak $\epsilon$-net for points \cite{Alon1992a}]\label{Theorem:WeakPoints}
    For any dimension $d\geq 1$ and $\epsilon>0$ there is $W=W(\epsilon,d)$ with the following property: For every finite (multi-)set $P$ of points in $\reals^d$ one can find $W$ points in $\reals^d$ that pierce every convex set $A\subseteq \reals^d$ with $|A\cap P|\geq \epsilon |P|$.
\end{theorem}

Understanding the asymptotic behaviour of $W(\epsilon,d)$ is one of the most challenging open problems in discrete geometry.

\medskip
The starting point of our investigation is the Colorful Helly Theorem of L\'aszl\'o Lov\'asz, first stated in \cite{Barany1982}, which concerns the scenario in which the intersecting $(d+1)$-tuples form a complete $(d+1)$-partite hypergraph.

\begin{definition}
    We say that a finite family of convex sets $\F$ is \emph{$k$-colored} if each set $K\in \F$ is colored with (at least) one of $k$ distinct colors. The $k$-coloring of $\F$ can be expressed by writing $\F$ as a union of $k$ \emph{color classes} $\F_1\cup \F_2\cup \dots \cup \F_{k}$, where each class $\F_i$ consists of the sets with color $i\in [k]$.
    We say that the $k$-colored family $\F$, with color classes $\F_1,\dots,\F_k$, has the \emph{Colorful Helly property, or $\CH(\F_1,\dots,\F_k)$} if every rainbow selection $K_i\in \F_i$, for $1\leq i\leq k$, has non-empty intersection $\bigcap_{i=1}^{k}K_i\neq \emptyset$.
\end{definition}

\begin{theorem}[Colorful Helly's Theorem] \label{theorem:CHT}
    Let $\F$ be a $(d+1)$-colored family of convex sets in $\reals^d$, with color classes $\F_1,\dots,\F_{d+1}$. Then $\CH(\F_1,\dots,\F_{d+1})$ implies that there is a color class $\F_i$ with non-empty intersection $\bigcap \F_i\neq \emptyset$.
\end{theorem}

Notice that Theorem \ref{theorem:CHT} says nothing about transversals to the remaining $d$ color classes $\F_j$, with $j\in [d+1]\setminus\{i\}$.
The primary goal of this paper is to gain a deeper understanding of the transversals to \emph{all} of the color classes $\F_i$ in a $(d+1)$-colored family $\F$ that satisfies $\CH$.

Theorem \ref{theorem:CHT} is in close relation, via point-hyperplane duality, with the colorful version of the Carathéodory theorem due to Bárány \cite{Barany1982}. Holmsen \emph{et al.} \cite{Holmsen2008} and independently Arocha \emph{et al.} \cite{Arocha2009} recently established the following strengthening of Bárány's result:

\begin{theorem}[Very Colorful Carathéodory Theorem]
    Let $P$ be a finite set of points in $\reals^d$ colored with $d+1$ colors. If every $(d+1)$-colorful subset of $P$ is separated from the origin, then there exist two colors such that the subset of all points of these colors is separated from the origin.
\end{theorem}

Unfortunately, there is no Very Colorful Helly Theorem which guarantees that a second color class can be pierced with few points, as is illustrated by the following example (see Figure \ref{fig:HellyOpt}). Let $\F_{d+1}=\{\reals^d\}$ and, for each $1\leq i\leq d$ let $\F_i$ be a collection of hyperplanes orthogonal to the $x_i$-axis. Then $\F_{d+1}$ is the only class that has a point transversal, moreover, each of the remaining classes may need an arbitrarily large number of points in order to be pierced. Note, though, that one can cross all the sets of $\bigcup_{i=1}^{d+1} \F_i$ by a \emph{single line}.

\begin{figure}
    \begin{center}
        \includegraphics{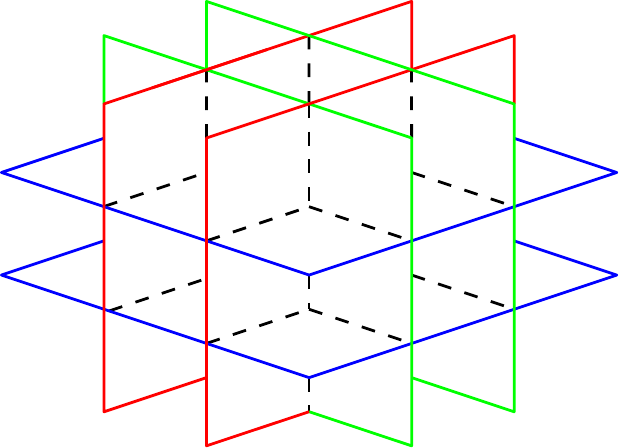}
        \caption{Optimality of the Colorful Helly Theorem in $\reals^3$. For each $1\leq i\leq 3$ the family $\F_i$ consists of $x_i$-orthogonal planes.}
        \label{fig:HellyOpt}
    \end{center}
\end{figure}

\subsection{Our results}
Our main result suggests that, in a sense, the scenario in Figure \ref{fig:HellyOpt} is the only possible unless an additional color class can be pierced by few points.

\begin{theorem}\label{theorem:main}
    For each dimension $d\geq 2$ there exist numbers $f(d)$ and $g(d)$ with the following property.
    Let $\F$ be a finite $(d+1)$-colored family of convex sets in $\R^d$ (with color classes $\F_1,\dots,\F_{d+1}$) that satisfies $\CH(\F_1,\dots,\F_{d+1})$. Let $i\in [d+1]$ be a color whose class $\F_i$ has a non-empty intersection (by Theorem \ref{theorem:CHT}).
    Then one of the following statements must also hold:
\begin{enumerate}
    \item an additional color class $\F_j$, for $j\in [d+1]\setminus \{i\}$ can be pierced by $f(d)$ points, or
    \item the entire family $\F$ can be crossed by $g(d)$ lines.
\end{enumerate}
\end{theorem}

Theorem \ref{theorem:main} is equivalent to the following statement concerning $d$-colored families of convex sets.

\begin{theorem}\label{theorem:main-d}
    For each dimension $d\geq 2$ there exist numbers $f'(d)$ and $g'(d)$ with the following property.
    Let $\F$ be a finite $d$-colored family of convex sets in $\R^d$, with color classes $\F_1,\dots,\F_d$, that satisfies $\CH(\F_1,\dots,\F_d)$. 
    Then one of the following statements holds:
    \begin{enumerate}
        \item there is a color class $\F_j$, for $j\in [d]$, that can be pierced by $f'(d)$ points, or
        \item the entire family $\F$ can be crossed by $g'(d)$ lines.
    \end{enumerate}
\end{theorem}

Theorem \ref{theorem:main} immediately follows from Theorem \ref{theorem:main-d} by setting $f(d)=f'(d)$ and $g(d)=g'(d)+1$. For the other direction, by letting $\F_{d+1}=\{\reals^d\}$ we can set $f'(d)=f(d)$ and $g'(d)=g(d)$.

Notice that in the $d$-colored scenario of Theorem \ref{theorem:main-d} one can use Theorem \ref{theorem:CHT} to obtain \emph{one} color class $\F_i$ that can be crossed by a single line (through a generic projection of $\F_d$ to $\reals^{d-1}$). The main strength of Theorem \ref{theorem:main-d} is that it shows a complementary relation between transversals to multiple colors $\F_i$, for $i\in [d]$. This relation can be further generalized as follows.

\begin{theorem}\label{theorem:structure}
    For all $1\leq i\leq d$ there exist numbers $f(i,d)$ and $g(i,d)$ with the following property.
    Let $\F$ be a finite $(d+1)$-colored family of convex sets in $\R^d$ that satisfies $\CH(\F_1,\dots,\F_{d+1})$. Then there exist $k\in[d]$ and a re-labeling of the color classes $\F_1,\dots,\F_{d+1}$ of $\F$ so that
    \begin{enumerate}
        \item $\bigcup_{1\leq j\leq k} \F_j$ can be pierced by $f(k,d)$ points, and
        \item $\bigcup_{k< j\leq d+1} \F_j$ can be crossed by $g(k,d)$ $k$-flats.
    \end{enumerate}
\end{theorem}

In other words, Theorem \ref{theorem:structure} characterizes the families of sets with the Colorful Helly property up to their transversal structure by flats.

This paper is organized as follows. In Section \ref{Sec:Main} we prove our main technical results -- Theorems \ref{theorem:main-d} and \ref{theorem:structure}. To this end, we establish a series of claims of independent interest that concern \emph{$2$-colored families} of convex sets. Despite the apparent weakness of the $2$-colored hypothesis in dimension higher than $2$, these results provide  all the essential ingredients for our analysis.
Theorem \ref{theorem:main-d} is finally established by repeatedly invoking a so called ``Step-Down'' Lemma which provides a crucial relation between $k$-flat and $(k-1)$-flat transversals of families with the Colorful Helly property, for all $1\leq k\leq d-1$.

The proof of the ``Step-Down'' Lemma is deferred to Section \ref{Sec:StepDown}, and it is based on a careful adaptation of the machinery of Alon and Kleitman \cite{Alon1992} and Alon and Kalai \cite{Alon1995}, to families of convex sets whose intersection graph is complete bi-partite.

Section \ref{Sec:Example} is devoted to constructing a lower bound for $g'$ in Theorem \ref{theorem:main-d}. Our example implies that, independently of the value given to $f'(d)$, $g'(d)\ge \lceil \frac {d+1}2\rceil$.

Finally, in Section \ref{Sec:Discus} we conclude the paper with several intriguing questions for future study.

\section{Proofs of Theorems \ref{theorem:main-d} and \ref{theorem:structure}}\label{Sec:Main}

A crucial ingredient of our proof is the following claim which concerns $2$-colored families.

\begin{lemma}\label{proposition:twocolors}
    Let $\A$ and $\B$ be families of convex sets in $\R^d$ so that $A\cap B\neq\emptyset$ for every $A\in \A$ and $B\in \B$. Then either
    \begin{enumerate}
        \item[(1)] $\bigcap \A\neq\emptyset$, or
        \item[(2)] $\B$ can be crossed by $d$ hyperplanes.
    \end{enumerate}
\end{lemma}

One can establish Theorem \ref{theorem:main-d} in dimension $d=2$ (with $f'(2)=1$ and $g'(2)\leq 4$) by applying Lemma \ref{proposition:twocolors} twice.
The weaker transversal guarantee of Lemma \ref{proposition:twocolors} in higher dimension $d\geq 3$ (namely, crossing by few hyperplanes instead of few lines) is due to the weaker, $2$-colored hypothesis.

\begin{proof}

    Assume that \emph{(1)} does not hold. Then by Helly's theorem there are convex sets $A_i\in\A$, for $1\leq i\leq d+1$, with empty intersection. By a standard argument (see e.g. \cite[Theorem 7.1]{Barany2014t}), there exist $d+1$ halfspaces $H_i\supseteq A_i$ with empty intersection. Let $\Pi_i$ be the bounding hyperplane of $H_i$. We claim that the union of the first $d$ hyperplanes $\Pi_i$ (for $i=1,\dots,d$) must meet all the sets from $B$. See Figure \ref{fig:Separate} for an illustration in $\reals^2$.
\begin{figure}
    \begin{center}
        \includegraphics[width=0.5\textwidth]{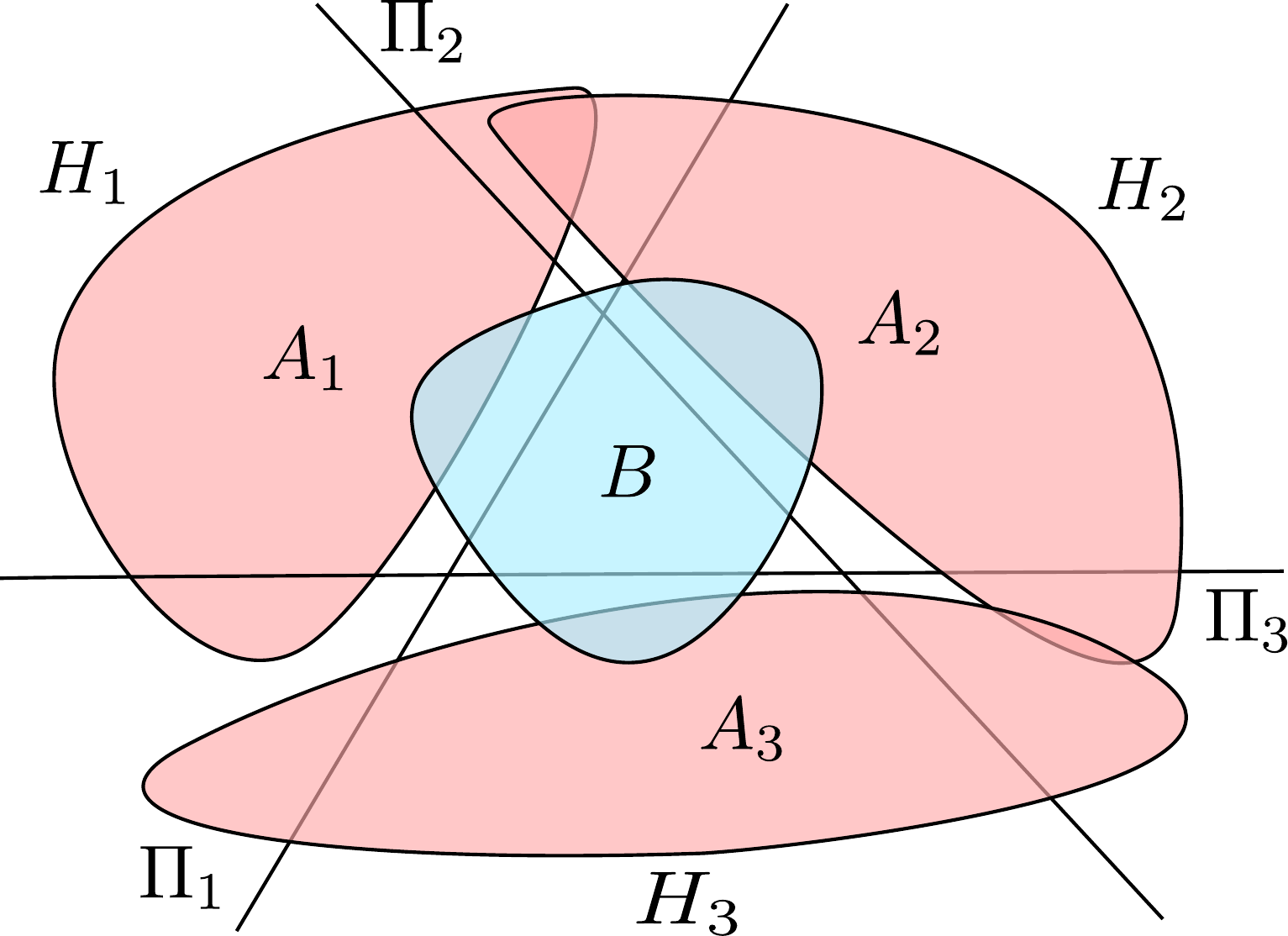}
				\caption{\sf \small Proof of Lemma \ref{proposition:twocolors}. We have $A_1,A_2,A_3\in \A$ and $B\in \B$. Since $A_1\cap A_2\cap A_3=\emptyset$, we have have halfspaces $H_i\supset A_i$, for $1\leq i\leq 3$, so that $\bigcap_{i=1}^n H_i=\emptyset$. Hence, the set $B\in \B$ must cross at least one of the respective bounding lines $\Pi_1$ and $\Pi_2$ of $H_1$ and $H_2$ to meet the sets $A_1$, $A_2$ and $A_3$. }
        \label{fig:Separate}
    \end{center}
\end{figure}
    
    Indeed, consider the arrangement of $\Pi_1, \dots, \Pi_d$ and suppose that a set $B\in \B$ does not intersect any of the hyperplanes $\Pi_i$. Then $B$ must be completely contained in an open cell $\sigma$ of their arrangement. Since $B$ intersects each of the sets $A_i$, for $1\leq i\leq d$, we obtain $\sigma=\bigcap_{i=1}^d H_i$. However, then $B$ cannot intersect $A_{d+1}\subset H_{d+1}$, since $H_1\cap\dots\cap H_{d+1}=\emptyset$. This contradiction implies \emph{(2)}.
\end{proof}

Both Theorems \ref{theorem:main-d} and \ref{theorem:structure} are established by iterating the following more refined variant of Lemma \ref{proposition:twocolors}.

\begin{lemma}[``Step-Down'' Lemma]\label{Lemma:Iterate}
    For any $1\le k\le d$ and $m\ge 1$ there exist numbers $F(m,k,d)$ and $G(m,k,d)$ with the following property.
    
    Let $\A$ and $\B$ be finite families of convex sets in $\reals^d$ so that the family 
    $$\I(\A,\B):=\{A\cap B\mid A\in \A,B\in \B\}$$
    can be crossed by $m$ $k$-flats. Then one of the following conditions is satisfied:
    \begin{enumerate}
        \item $\A$ can be pierced by $F(m,k,d)$ points, or
        \item $\B$ can be crossed by $G(m,k,d)$ $(k-1)$-flats.
    \end{enumerate}
\end{lemma}

Notice that the hypothesis of Lemma \ref{Lemma:Iterate} implies, in particular, that every two sets $A\in \A,B\in \B$ intersect.
Thus, Lemma \ref{proposition:twocolors} deals with the special case of Lemma \ref{Lemma:Iterate} in which $k=d$, yielding $F(1,d,d)=1$ and $G(1,d,d)\leq d$.

We defer the somewhat complex proof of Lemma \ref{Lemma:Iterate} to Section \ref{Sec:StepDown}. It combines the standard duality relation between transversal and packing numbers of hypergraphs with a ``hyperplane'' variant of Theorem \ref{Theorem:WeakPoints}, due to Alon and Kalai \cite{Alon1995}, in which we are given 
a collection of hyperplanes $H$ and seek to find a small hyperplane transveral to all the convex sets that are crossed by a fixed fraction of the hyperplanes of $H$.

We are now ready to establish Theorem \ref{theorem:main-d}.

\begin{proof}[Proof of Theorem \ref{theorem:main-d}.]
    Let $\F$ be a $d$-colored family that satisfies $\CH(\F_1,\dots,\F_d)$ and does not satisfy conclusion $1$.
    Since the labeling of the color classes $\F_1,\dots,\F_d$ is arbitrary, it suffices to show that the last family $\F_d$ can be crossed by few lines.
    
    The underlying idea of our analysis is as follows.
    We apply the ``Step-Down'' Lemma \ref{Lemma:Iterate} $d-1$ times. In the $i$-th iteration (for $1\leq i\leq d-1$) we deal with a $(d-i+1)$-colored and \emph{essentially $(d-i+1)$-dimensional} scenario in which the family of all the $(d-i+1)$-wise intersections
    $$\I(\F_i,\dots,\F_d):=\bigg\{\bigcap_{j=i}^d A_j\mid A_j\in \F_j\bigg\}$$
    is ``captured'' by only $M=M(i,d)$ copies of $\reals^{d-i+1}$ within $\reals^d$. Unless $\F_i$ can be pierced by $F(M,i,d)$ points, the ``Step-Down'' Lemma can be used to further reduce the intrinsic ``transversal dimension'' of the remaining sets $\F_{i+1},\dots,\F_d$ to $d-i$.
    
    For reasons that will become evident shortly, we set
    $$M(i,d):=
    \begin{cases}
        1 & \text{for }i=1,\\
        d & \text{for }i=2,\\
        G(M(i-1,d),d-i+2,d) & \text{for }3\leq i\leq d-1.
    \end{cases}$$
    
    For $i=1$, the condition that $\I(\F_1,\dots,\F_d)$ is crossed by $\reals^d$ is equivalent to the $\CH(\F_1,\dots,\F_d)$ hypothesis. Notice that the families $\A:=\F_1$ and $\B:=\I(\F_2,\dots,\F_{d})$ satisfy the hypothesis of Lemma \ref{proposition:twocolors}. Therefore, unless $\F_1$ can be pierced by a single point, the family $\I(\F_2,\dots,\F_d)$ can be crossed by $M(2,d)=d$ hyperplanes.
    
    Let us now fix $2\leq i\leq d-1$ and assume that $\I(\F_i,\dots,\F_d)$ can be crossed by $M=M(i,d)$ $(d-i+1)$-flats. Note that the families $\A:=\F_i$ and $\B:=\I(\F_{i+1},\dots,\F_d)$ satisfy the $2$-colored hypothesis of Lemma \ref{Lemma:Iterate}. Therefore, given that $\F_i$ cannot be pierced by $F(M,d-i+1,d)$ points, the other family $\I(\F_{i+1},\dots,\F_d)$ can be crossed by $M(i+1,d)=G(M,d-i+1,d)$ $(d-i)$-flats.
    
    Assuming neither of the families $\F_i$, for $1\leq i\leq d-1$, can be pierced by $F(M(i,d),d-i+1,d)$ points, by the end of the $(d-1)$-st iteration we can cross the last color class $\F_d$ by $G(M(d-1,d),2,d)$ lines.
    
    This proves Theorem \ref{theorem:main-d} with
    \begin{align*}
        f'(d)&=\max\{F(M(i,d),d-i+1,d)\mid 1\leq i\leq d-1\}, \text{ and}\\
        g'(d)&=d\cdot G(M(d-1,d),2,d).
    \end{align*}
\end{proof}

\begin{remark} In the proof of Theorem \ref{theorem:main-d}, the value of $g'(d)$ can be further improved to
    $$g'(d)=(d-1)\cdot G(M(d-1,d),2,d)+1$$
    by observing that at least one of the families $\F_1,\dots,\F_d$ can be crossed by a \emph{single} line.
    To this end, we project $\F$ in a generic direction $\vec{\nu}$ and apply Theorem \ref{theorem:CHT} to the resulting $d$-colored family $\F(\vec{\nu})=\F_1(\vec{\nu})\cup\dots \cup \F_d(\vec{\nu})$ within $\reals^{d-1}$. This yields an intersecting color class $\F_i(\vec{\nu})$ within $\reals^{d-1}$ and, therefore, a $\vec{\nu}$-parallel line which crosses the respective color class $\F_i$.
\end{remark}

\begin{proof}[Proof of Theorem \ref{theorem:structure}]
The Theorem is obviously true for $d=1$ (with $f(1,1)=1$, $g(1,1)=1$).
    Assume with no loss of generality that the last color class $\F_{d+1}$ can be pierced by a point (in accordance with Theorem \ref{theorem:CHT}).
    We adopt the notation of the previous proof while dealing with the remaining color classes $\F_1,\dots,\F_d$.
    
    Let $l$ be the size of the largest sequence $j_1,j_2\dots,j_l$ so that no class $\F_{j_i}$ can be pierced by $F(M(l,d),d-l+1,d)$ points.
    Let $\F'$ be the relabeling of $\F$ whose first
    $l$ color classes satisfy $\F'_{i}=\F_{j_i}$, for $1\leq i\leq l$. By following the first $l-1$ iterations of the proof of Theorem \ref{theorem:main-d}, we obtain that $\F'_l=\F_{j_l}$ can be crossed by $G(M(l-1,d),d-l,d)$ $(d-l+1)$-flats. By reordering of $j_1,\dots,j_l$, this establishes the claim of Theorem \ref{theorem:structure} for $\F$ with
    \begin{align*}
        k&=d-l+1,\\
        f(k,d)&=k\cdot F(M(d-k+1,d),k,d),\text{ and}\\
        g(k,d)&=(d-k+1)\cdot G(M(d-k,d),k+1,d).\qedhere
    \end{align*}
\end{proof}

\section{Proof of the ``Step-Down'' Lemma}\label{Sec:StepDown}

We develop a bi-partite variant of the machinery that was used by Alon and Kleitman \cite{Alon1992} to establish the $(p,q)$-Conjecture (Theorem \ref{theorem:PQ}). This method was extended by Alon and Kalai \cite{Alon1995} to obtain an analogous result for hyperplane transversals.

\subsection{From piercing to packing numbers}

The crucial ingredient of Alon-Kleitman approach was a duality relation between transversal (or piercing), and packing (or matching) numbers of hypergraphs.

\begin{definition}
    Let $\H=(\V,\E)$ be a hypergraph, where $\V$ is a finite set of elements and $\E$ is a family of subsets of $\V$. The elements of $\V$ are called \emph{vertices}, and the sets of $\E$ are called \emph{edges}.
\end{definition}

A subset $A\subset \V$ is a \emph{transversal} for $\H$ if it intersects every edge $S\in \E$ (i.e., $A\cap S\neq \emptyset$ for each $S\in \E$).
The \emph{transversal number} $\tau(\H)$ of $\H$ is the size $|A|$ of the smallest such transversal $A$.

A non-negative function $f:\V\rightarrow \reals$ is a \emph{fractional transversal} for $\H$ if it satisfies $\sum_{x\in S}f(x)\geq 1$ for every edge $S\in \E$. The \emph{fractional transversal number} $\tau^*(\H)$ of $\H$ is the total ``weight'' $\sum_{x\in \V}f(x)$ of the ``lightest'' fractional transversal $f$ of $\H$ (that is, it is the smallest possible value $\sum_{x\in \V}f(x)$ that can be attained by a fractional tranfsversal $f$).

A subset of edges $\E'\subseteq \E$ is called a \emph{$b$-matching} (or \emph{$b$-packing}) for $\H$ if every vertex $x\in \V$ belongs to at most $b$ edges of $\E'$. The \emph{$b$-matching number} $\nu_b(\H)$ of $\H$ is the size $|\E'|$ of the largest such $b$-matching $\E'$.

A non-negative function $g:\E\rightarrow \reals$
is a \emph{fractional matching} for $\H$ if it satisfies
$$\sum_{S\in \E: x\in S}g(S)\leq 1$$
for every $x\in \V$. The \emph{fractional matching number $\nu^*(\H)$} of $\H$ is the total ``weight'' $\sum_{S\in \E}g(S)$ of the ``heaviest'' fractional matching $g$ of $\H$ (that is, it is the largest possible value $\sum_{S\in \E}g(S)$ that can be attained by a fractional matching $g$).

A standard use of Linear Programming duality \cite{Alon1992,Alon1995,Alon2002} yields the following relation between transversal and matching numbers of $\H$.

\begin{theorem}\label{Theorem:Dual}
    We have
    $$\nu_b(\H)/b\leq \nu^*(\H)=\tau^*(\H)\leq \tau(\H)$$
    for every hypergraph $\H$ and $b\geq 1$.
\end{theorem}

The proof of Theorem \ref{theorem:PQ} by Alon and Kleitman \cite{Alon1992} combines the following key elements:
\begin{itemize}
    \item An abstract hypergraph $\H_0(\F)$, whose edges correspond to the sets of $\F$, is constructed. Each vertex of $\H_0(\F)$ is a point that pierces some sub-family $\F'\subset \F$. (To keep the vertex set finite, we add \emph{one} vertex for each $\F'\subset \F$ with non-empty intersection $\bigcap \F'\neq \emptyset$.)
    \item The fractional matching number $\nu^*(\H_0(\F))=\tau^*(\H_0(\F))$ is bounded  from above using a suitable fractional Helly-type result (Theorem \ref{theorem:Fractional}).
    \item The fractional transversal for $\H_0(\F)$ is converted to an integral one using a weak $\epsilon$-net result for point transversals \cite{Alon1992a}.
\end{itemize}

\subsubsection{Overview.} As we cast the $2$-colored setup of the ``Step-Down'' Lemma into the above abstract framework, several fundamental challenges are to be addressed.

As we seek a \emph{relation} between the transversal numbers of $\A$ and $\B$, we maintain \emph{two hypergraphs} $\H_0(\A)$ and $\H_{k-1}(\B)$, where the former (resp., latter) hypergraph describes partial point (resp., $(k-1)$-flat) transversals to $\A$ (resp., $\B$).
To show that at least one of $\H_0(\A)$ and $\H_{k-1}(\B)$ has a bounded fractional packing number, we need a suitable fractional Helly-type result which is conveniently provided by the fractional variant of our $2$-colored Lemma \ref{proposition:twocolors}.
Finally, to convert a fractional transversal for $\H_{k-1}(\B)$ into an integral one, we need a small-size weak $\epsilon$-net construction for $(k-1)$-flats.

Unfortunately, no Helly-type results and no weak $\epsilon$-net constructions are known for transversals by general $(k-1)$-flats in $\reals^d$, unless $k=1$ \cite{Alon1992} or $k=d$ \cite{Alon1995}.
Note though that, in the scenario of Lemma \ref{Lemma:Iterate}, the pairwise intersections $\I(\A,\B)$ are assumed to ``occur'' within few $k$-dimensional flats of $\reals^d$. 
We can therefore invoke the fractional variant of Lemma \ref{proposition:twocolors} in dimension $k$ and similarly apply the weak $\epsilon$-net construction of Alon and Kalai \cite{Alon1995} for hyperplanes in $\reals^k$.

\subsection{Bounding the fractional packing number}
\label{sec:bic}

Let $\A$ and $\B$ be families of convex sets that satisfy the hypothesis of Lemma \ref{Lemma:Iterate}. That is, the family $\I(\A,\B)$ of pairwise intersections can be crossed by $m$ $k$-flats $\Gamma_1,\dots,\Gamma_m$.

\subsubsection{The hypergraphs $\H_0(\A)$ and $\H_{k-1}(\B)$.} 
Below we define the  abstract hypergraphs $\H_0(\A)$ and $\H_{k-1}(\B)$ which describe, respectively, partial point transversals to $\A$, and partial transversals by $(k-1)$-flats to $\B$.

The hypergraph $\H_0(\A)=(\V_\A,\E_\A)$ is constructed analogously to the one of Alon and Kleitman \cite{Alon1992}: For every subfamily $\A'\subset \A$ with $\bigcap \A'\neq \emptyset$ we add a point $x_{\A'}\in \bigcap \A'$ to $\V_\A$, and for every convex set $A\in \F$ we add the edge $e_A:=\{x_{\A'}\mid \bigcap \A'\neq \emptyset, A\in \A'\}$ to $\E_\A$.

The definition of $\H_{k-1}(\B)=(\V_\B,\E_\B)$ is somewhat more involved: For every subfamily $\B'\subset \B$ that can be crossed by a $(k-1)$-flat \emph{within} $\bigcup_{i=1}^m\Gamma_i$, we add one such $(k-1)$-flat $\sigma_{\B'}\subset \bigcup_{i=1}^m\Gamma_i$ to $\V_\B$. Accordingly, each $B\in \B$ yields the edge 
$$e_{B}:=\{\sigma_{\B'}\mid B \in \B'\}\in \E_\B.$$

To show that \emph{at least one} of the hypergraphs $\H_0(\A)$ or $\H_{k-1}(\B)$ has a bounded fractional packing number, we use the following fractional variant of our $2$-colored Lemma \ref{proposition:twocolors}.

\begin{lemma}[Fractional $2$-colored Lemma]\label{theorem:fractwo}
For every $0<\alpha\leq 1$ and $d\ge 1$ there exist $\gamma=\gamma(\alpha,d)$ and $\lambda=\lambda(\alpha,d)$ with the following property. 
Let $\A$ and $\B$ be finite (multi-)families of convex sets in $\R^d$ so that $A\cap B\neq\emptyset$ holds for at least $\alpha|\A||\B|$ of the pairs $A\in \A$ and $B\in \B$. Then either
	\begin{enumerate}
		\item one can pierce at least $\gamma|\A|$ members of $\A$ by a single point, or 
		\item one can cross at least $\lambda|\B|$ members of $\B$ by a single hyperplane.
	\end{enumerate}
\end{lemma}

\begin{proof}
    We establish the lemma with $$\lambda:= \frac{\alpha^{d+2}}{4d\cdot 3^{d+1}} \quad \text{ and } \quad \gamma:= \min\left(\beta(d\lambda,d), \frac{\alpha}{6d}\right),$$
    
    \noindent where the function $\beta(\cdot,\cdot)$ is defined as in the Fractional Helly Theorem \ref{theorem:Fractional}. The reasons behind this choice will become evident during the proof.
    
    We may assume $|\A| \geq \frac{6d}{\alpha}$, for otherwise $\gamma|\A|\leq 1$ and the result follows immediately.
    
    For a subset $\A'\in \binom {\A}{d+1}$, and $B\in \B$, we say that $(\A',B)$ is a {\it special pair} if $B$ intersects every set in $\A'$; in other words, $\A'$ and $B$ form a star in the bipartite graph that represents pairwise intersections between the elements of $\A$ and $\B$.
    
    Let $T$ denote the set of all the special pairs $(\A',B)$ as above.
    We first establish a lower bound for the cardinality of $T$. To this end, we claim that there are at least $\frac{\alpha}{2}|\B|$ \textit{heavy} elements $B$ of $\B$ each of which intersects with at least $\frac{\alpha}{2}|\A|$ elements of $\A$. Indeed, otherwise we contradict the hypothesis as the number of pairwise intersections would fewer than $$|\A|\left(\frac{\alpha}{2}|\B|\right)+\left(\frac{\alpha}{2}|\A|\right)|\B|= \alpha|\A||\B|.$$
        
    Let $a=\left\lceil \frac{\alpha}{2}|\A| \right\rceil$ and $b=\left\lceil \frac{\alpha}{2}|\B| \right\rceil$. The discussion above shows that there are at least $b$ heavy elements, each of which appears in at least $\binom{a}{d+1}$ special pairs. Therefore:
        \begin{equation}\label{eq:t1t2}
            |T|\geq b\binom{a}{d+1}\geq \left(\frac{\alpha}{2}|\B|\right) \left(\frac{\alpha}{3}\right)^{d+1}\binom{|\A|}{d+1}=2d\lambda|\B|\binom{|\A|}{d+1}.
        \end{equation}
        
        \noindent The second inequality is obtained as follows, where we use $|\A| \geq \frac{6d}{\alpha}$ at the end:
        
        \begin{equation*}
            \frac{\binom{a}{d+1}}{\binom{|\A|}{d+1}}=\frac{a}{|\A|}\cdot\frac{a-1}{|\A|-1}\cdot \ldots \cdot \frac{a-d}{|\A|-d}\geq\left(\frac{a-d}{|\A|-d}\right)^{d+1}\geq\left(\frac{\frac{\alpha}{2}|\A|-d}{|\A|-d}\right)^{d+1}\geq\left(\frac{\alpha}{3}\right)^{d+1}.
        \end{equation*}
    
    \medskip
    Now, consider the subdivision $T=T_1\uplus T_2$:
    
        \begin{align*}
    T_1:= & \left\{(\A',B)\in T\mid \bigcap_{A\in \A'} A \neq \emptyset\right\}\\
    T_2:= & \left\{(\A',B)\in T\mid \bigcap_{A\in \A'} A=\emptyset\right\}.
      \end{align*}

        If at least $d\lambda\binom{|\A|}{d+1}$ of the $(d+1)$-subfamilies of $\A$ are intersecting, then by the Fractional Helly Theorem \ref{theorem:Fractional} we obtain an intersecting subfamily of $\A$ of size $\gamma |\A|$ and we are done. Therefore, we may assume that less than $d\lambda\binom{|\A|}{d+1}$ of the $(d+1)$-subfamilies of $\A$ are intersecting. Since each of them appears in at most $|\B|$ special pairs of $T_1$, we obtain
        \begin{equation}\label{eq:t1}
            |T_1| < d\lambda|\B|\binom{|\A|}{d+1}.
      \end{equation}
        
        Equations \ref{eq:t1t2} and \ref{eq:t1} imply that $|T_2|\geq d\lambda|\B|\binom{|\A|}{d+1}$. By the pigeon-hole principle there is a non-intersecting $(d+1)$-subfamily $\A_0\subset \A$ that appears in at least $d\lambda|\B|$ special pairs. Let $\B_0$ be the family of all the elements $B$ in $\B$ which yield such a special pair $(\A_0,B)$. Applying Lemma \ref{proposition:twocolors} to $\A_0$ and $\B_0$ we get a collection of $d$ hyperplanes that cross all the sets in $B_0$. Therefore, again by the pigeon-hole principle, one of these hyperplanes crosses at least $\frac{1}{d}|\B_0|\geq\lambda|\B|$ of the sets of $\B$. 
    \end{proof}

Now we prove the following auxiliary statement.

\begin{claim}\label{Claim:Packing}
    We have that either $\nu^*(\H_0(\A))\leq 1/\left(\gamma(1/m,k)\right)$ or $\nu^*(\H_{k-1}(\B))\leq 1/\left(\lambda(1/m,k)\right)$, where $m$, $\H_0(\A)$ and $\H_{k-1}(\B)$ are as defined above, and the functions $\gamma$ and $\lambda$ are defined as in Lemma \ref{theorem:fractwo}.
\end{claim}

\begin{proof}
The fractional packing and fractional transversal numbers exist as we are optimizing continuous functions on a compact set. Moreover, the optimal value may be obtained via a rational approximation. Thus, given the contrapositive assumption, we have a pair of non-negative rational assignments $f: \E_\A\rightarrow \Q$ and $g: \E_\B\rightarrow \Q$ so that the following inequalities hold for all $x_0\in \V_\A$ and $\sigma_0\in \V_\B$:
\begin{align}
    \sum_{x_0\in e}f(e)<&\gamma(1/m,k) \sum_{e\in \E_\A}f(e),\label{Eq:PackingA}\\
    \sum_{\sigma_0\in e}g(e)<&\lambda(1/m,k)\sum_{e\in \E_\B}g(e).\label{Eq:PackingB}
\end{align}

By scaling $f$ and $g$, we end up with a pair of integer functions $f:\E_\A\rightarrow \ZZ^+$ and $g:\E_\B\rightarrow \ZZ^+$ which still satisfy the Inequalities \eqref{Eq:PackingA} and \eqref{Eq:PackingB}. By the definition of $\H_0(\A)$ and $\H_{k-1}(\B)$, this yields a pair of multisets $\hat{\A}$ and $\hat{\B}$ of, respectively, $\A$ and $\B$, so that
\begin{enumerate}
    \item[(i)] no point in $\reals^d$ crosses more than $\gamma(1/m,k)|\hat{\A}|$ members of $\hat{\A}$, and
    \item[(ii)] no $(k-1)$-flat within $\bigcup_{i=1}^m \Gamma_i$ crosses more than $\lambda(1/m,k)|\hat{\B}|$ members of $\hat{\B}$.
\end{enumerate}

By the pigeonhole principle, one of the $k$-flats $\Gamma_i$ must cross at least $(1/m)|\I(\A,\B)|$ of the pairwise intersections $\I(\A,\B)$. Applying Lemma \ref{theorem:fractwo} to the cross-sections $\{A\cap \Gamma_i\mid A\in \hat{\A}\}$ and $\{B\cap \Gamma_i\mid B\in \hat{\B}\}$ within $\Gamma_i\cong \reals^k$, and with $\alpha:=1/m$, yields the eventual contradiction to the above properties (i) and (ii) of $\hat{\A}$ and $\hat{\B}$.
\end{proof}

\subsection{Wrap-up} Combining Claim \ref{Claim:Packing} with Theorem \ref{Theorem:Dual}, we obtain that at least one of the graphs $\H_0(\A)$ and $\H_{k-1}(\B)$ has a bounded \emph{fractional} transversal number, so one of the following inequalities must hold:
$$\tau^*(\H_0(\A))\leq \frac{1}{\gamma(1/m,k)},\qquad
\tau^*(\H_{k-1}(\B))\leq \frac{1}{\lambda(1/m,k)}.$$

Analogously to the proof of Claim \ref{Claim:Packing}, we obtain respectively either a rational (and not everywhere zero) function $f:\V_{\A}\rightarrow \Q^+$ so that every edge $e\in \E_\A$ (representing some set $A\in \A$) contains vertices (i.e., points) of total weight
$$\sum_{x\in e}f(x)\geq \gamma(1/m,k) \sum_{x\in \V_\A}f(x),$$
or a similar function $g:\V_{\B}\rightarrow \Q^+$ so that every edge $e\in \E_\B$ contains vertices of total weight
$$\sum_{\sigma\in e}g(\sigma)\geq \lambda(1/m,k)\sum_{\sigma\in \V_\B}g(\sigma).$$

Arguing as in the proof of Claim \ref{Claim:Packing}, we obtain either (i) a multiset of points $\hat{\V}_\A\subset \reals^d$ so that any member $A$ of $\A$ contains at least $\gamma(1/m,k)|\hat{\V}_\A|$ of these points, or (ii) a multiset $\hat{\V}_\B$ of $(k-1)$-flats within $\bigcup_{i=1}^m\Gamma_i$ so that any member $B$ of $\B$ is crossed by at least $\lambda(1/m,k)|\hat{\V}_\B|$ of the flats.

In the former case, we use Theorem \ref{Theorem:WeakPoints} to show that, in case (i), the family $\A$ can be pierced by
$$F(m,k,d):=W\left(\gamma(1/m,k),0,d\right)$$
points.

In the remaining case (ii), we use the following analogue of Theorem \ref{Theorem:WeakPoints} for hyperplane transversals, due to Alon and Kalai \cite{Alon1995}: 

\begin{lemma}[Weak $\epsilon$-net for hyperplanes]\label{Lemma:WeakPlanes}
    For any dimension $d\geq 1$ and $\epsilon>0$ there is $W_{hpl}(\epsilon,d)$ with the following property: For every finite (multi-)set $H$ of hyperplanes in $\reals^d$ one can find $W_{hpl}(\epsilon,d)$ hyperplanes in $\reals^d$ whose union crosses every convex set $A\subseteq \reals^d$ that meets at least $\epsilon |H|$ of the hyperplanes of $H$.
\end{lemma}

For each $1\leq i\leq m$ we apply Lemma \ref{Lemma:WeakPlanes} to construct a weak $(\lambda(1/m,k)/m)$-net with respect to the $(k-1)$-flats $\sigma\in \hat{\V}_\B$ that are contained in $\Gamma_i\cong\reals^k$. It is immediate to check that the resulting family of at most 
$$G(m,k,d):=m\cdot W_{hpl}\left(\frac{\lambda(1/m,k)}{m},k-1,k\right)$$
$(k-1)$-flats crosses each $B\in \B$: Since $B$ is crossed by at least $\lambda(1/m,k)|\hat{\V}_{\B}|$ $(k-1)$-flats of $\hat{\V}_{\B}$, and at least $(1/m)\lambda(1/m,k)|\hat{\V}_{\B}|$ of such flats must be contained in some $k$-flat $\Gamma_i$, then $B$ must be crossed by the corresponding $k$-dimensional net. $\Box$

\section{A lower bound for Theorem \ref{theorem:main-d}}\label{Sec:Example}

\begin{theorem}\label{Thm:LowerBound}
For every $d\geq 2$ and integer $f\geq 1$ there exists a $d$-colored family $\F=\F_1\uplus\F_2\uplus\ldots \uplus\F_d$ in $\reals^d$ that satisfies $\CH(\F_1,\F_2,\ldots,\F_d)$ and the following additional properties:
\begin{itemize}
    \item For every $1\le i\le d$, one needs at least $f$ points to pierce the color class $\F_i$. (In other words, $\tau(\G(\F_i))\geq f$.)
    \item At least $\lceil\frac{d+1}{2}\rceil$ lines are necessary to cross $\bigcup \F_i$.
\end{itemize}
\end{theorem}

We prove the result in the following two subsections. We begin with the case $d=2$ which is later used to deal with the general case.

\subsection{The planar construction}
Let $m=2f$ and $T_0$ be a triangle in the plane so that its bottom side is parallel to the $x$-axis. We first construct $m$ triangles $T_1,\ldots,T_m$, each with one horizontal side and vertices in the relative interiors of the three sides of $T_0$, and such that no three of these triangles $T_i,T_j,T_k$ for $1\leq i<j<k\leq m$ have a common intersection. A way to do this is to construct them recursively: we start with two arbitrary such triangles $T_1$ and $T_2$ and at each step $i>2$ we place the horizontal side of $T_i$ sufficiently close to the horizonal side of $T_0$ so that it avoids all previous pairwise intersections (see Figure \ref{fig:Previous}). Let the first color class $\F_1$ be the resulting family $\{T_1,\ldots,T_m\}$. Clearly we need at least $m/2=f$ points to pierce $\F_1$.

\begin{figure}
    \begin{center}
        \includegraphics{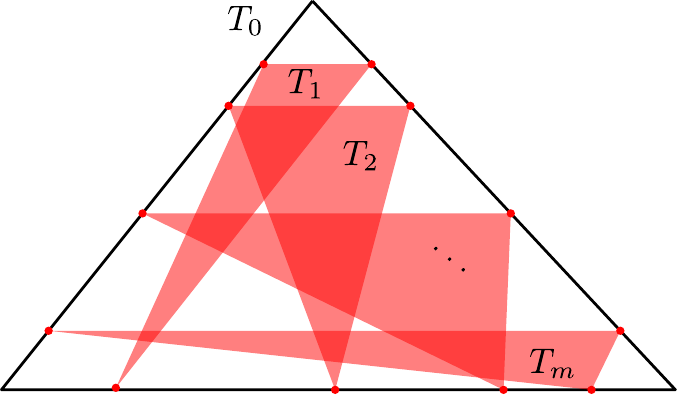}
				\caption{\sf \small The planar construction (for $m=4$). Each triangle $T_i$ has a horizontal topmost side which lies below all the pairwise intersections of $T_1,\ldots T_{i-1}$.}
        \label{fig:Previous}
    \end{center}
\end{figure}

Let $E_1,E_2,E_3$ be the three sides of $T_0$. As each set of $\F_1$ intersects the relative interior of each $E_i$, for $1\leq i\leq 3$, we can slightly shrink each $E_i$ away from its adjacent vertices of $T_0$ while preserving the intersection with every element of $\F_1$. The family $\F_2$ will consist of $m$ slightly translated copies of each (previously shrunk) segment $E_i$ so that they still intersect every triangle in $\F_1$ but are still pairwise disjoint. Note that we need at least $3m>f$ points to pierce $\F_2$.

In order to cross $\F_1\cup\F_2$ with lines, we need in particular to cross the interiors of $E_1,E_2,E_3$, so at least $2$ lines are needed.

\subsection{The general construction}
Set $d>2$ and $m=2f$. Let $\Delta^{(d)}\subset\reals^d$ be a $d$-simplex with vertex set $V=\{v_1,v_2,\ldots,v_{d+1}\}$. For each $1\le i\le d-1$, define $\tau_i$ to be the triangle with vertices $\{v_i,v_{i+1},v_{i+2}\}$.
As in the planar case, let $\T_i$ be a family of $m$ triangles, each with vertices in the relative interiors of the three sides of $\tau_i$, such that no three of them intersect.
Let $\hat{\F}_i^{(d)}$ be the family consisting of the sets $$\conv((V\setminus\{v_i,v_{i+1},v_{i+2}\})\cup \tau),$$
with $\tau\in \T_i$.
Let $\hat{\F}_d^{(d)}$  denote the family of all the $(d-1)$-dimensional faces (facets) of $\Delta^{(d)}$; see Figure \ref{fig:Tetrahedron}.

\begin{figure}
    \begin{center}
        \includegraphics{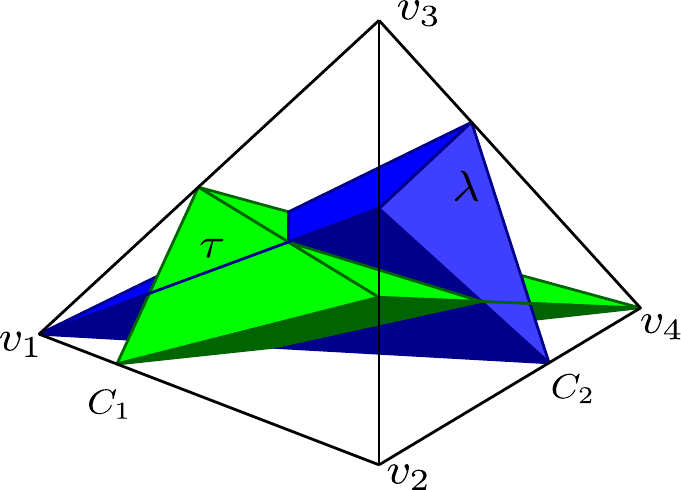}
        \caption{\sf \small The construction of $\hat{\F}^{(3)}$ -- a pair of sets $C_1\in \hat{\F}_1^{(3)}$ and $C_2\in  \hat{\F}_2^{(3)}$ are depicted. We have $C_1=\conv(\tau,v_4)$ and $C_2=\conv(\lambda,v_1)$, with $\tau\in \T_1$ and $\lambda\in \T_2$. The sets of $\hat{\F}^{(3)}_3$ are the facets of the bounding simplex $\Delta^{(3)}$.}
        \label{fig:Tetrahedron}
    \end{center}
\end{figure}

As the resulting $d$-colored family $\hat{\F}^{(d)}=\bigcup\hat{\F}^{(d)}_i$ can obviously be pierced by $d+1$ points, the convex sets in $\hat{\F}$ have to be suitably shrunk in order to satisfy the conditions of Theorem \ref{Thm:LowerBound}. However, before we describe the actual family $\F^{(d)}$, we establish a key property of the families $\hat{\F}^{(d)}_i$.

\begin{lemma}\label{lem:relint}
    For any selection of $C_i\in \hat{\F}_i^{(d)}$ with $2\le i\le d$ we have
    $$\left(\bigcap_{i=1}^{d-1} C_i\right) \cap \relint(C_d) \neq \emptyset,$$
    where $\relint(C)$ denotes the relative interior of $C$.
\end{lemma}

\begin{proof}
    We proceed by induction on the dimension $d$. For $d=2$ we define the families in a similar way as for $d>2$. Then, when $d=2$, the colored family $\hat{\F}^{(1)}$ is essentially the same as in the planar case, where by definition each triangle in $\F_1$ intersects the relative interiors of the sides of $T\simeq\tau_1$, which are precisely the elements of $\F_2$.
    
    Now assume that $d>2$ and the statement is true in dimension $d-1$.
    Note that the cross-sections of $\hat{\F}^{(d)}_1\cup\dots\cup\hat{\F}^{(d)}_{d-2}$ with the hyperplane $\pi_{d-1}$ spanned by $v_1,\ldots,v_{d}$ form the first $d-2$ color classes $\hat{\F}^{(d-1)}_1\cup\dots\cup\hat{\F}^{(d-1)}_{d-2}$ of the $(d-1)$-dimensional family $\hat{\F}^{(d-1)}$ (constructed with respect to $\Delta^{(d-1)}=\conv(v_1,\ldots,v_d)$ in $\pi_{d-1}\simeq \reals^{d-1}$); the last $(d-1)$-th color $\hat{\F}^{(d-1)}_{d-1}$ is composed of all the $(d-2)$-dimensional faces of $\conv(\{v_1,\dots,v_d\})$. Therefore, by induction, for every $(d-2)$-colorful selection of $C_1\in \hat{\F}_1^{(d)}, C_2\in \hat{\F}_2^{(d)},\ldots ,C_{d-1}\in \hat{\F}_{d-2}^{(d)}$ and a $(d-2)$-simplex $\sigma\in \bd \left(\Delta^{(d-1)}\right)$ we obtain
    $$\left(\bigcap_{i=1}^{d-2} C_i\right) \cap \relint(\sigma) \neq \emptyset.$$

\begin{figure}
    \begin{center}
        \includegraphics[width=0.85\textwidth]{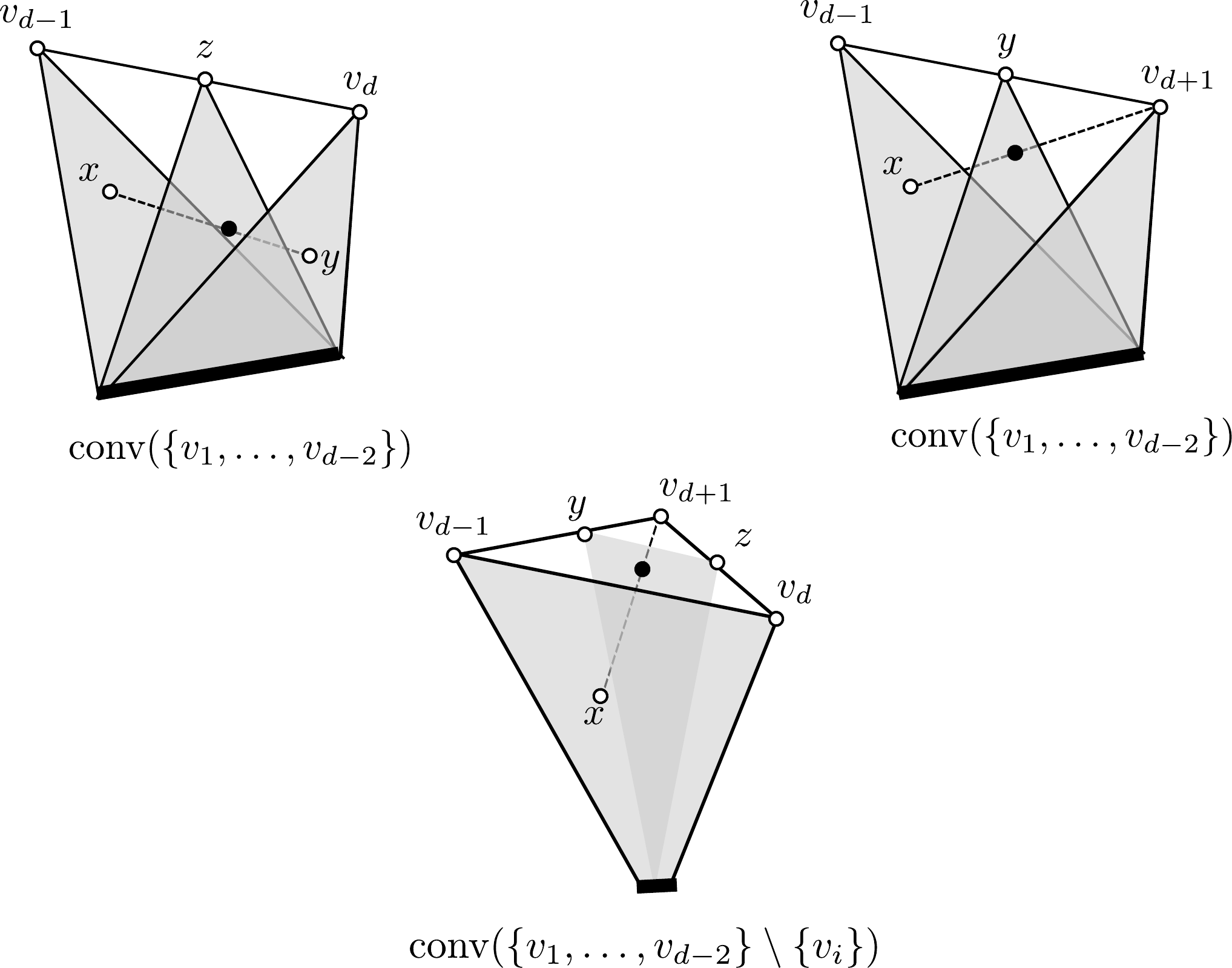}
        \caption{\sf \small Proof of Lemma \ref{lem:relint}. The three main cases are depicted. Left: In the first case, $C_d=\conv(V\setminus \{v_{d+1}\})$. Right:  $C_d=\conv(V\setminus \{v_d\})$. Bottom: $C_d=\conv(V\setminus \{v_i\})$ for $1\leq i\leq d-2$.}
        \label{fig:Cases}
    \end{center}
\end{figure}

   Consider a $d$-colorful choice $C_i\in \hat{\F}_i^{(d)}$ with $1\le i\le d$. In order to show that $\bigcap_{i=1}^{d-1} C_i$ intersects the relative interior of $C_d$, we distinguish between three cases.  In each case, we use the induction hypothesis to pick a pair of points $\bigcap_{i=1}^{d-2}C_i$ on different faces of $C_d$ which span an open segment $s$ in the relative interior of $C_d$. We then use the definition of $C_{d-1}$ to argue that it must intersect $s$. See Figure \ref{fig:Cases}.
    
    \begin{enumerate}
        \item If $C_d=\conv(V\setminus \{v_{d+1}\})$, by the induction hypothesis we know that the intersection $\bigcap_{i=1}^{d-2} C_i$ has points in the interiors of the facets $\conv(\{v_1,\ldots,v_{d-2},v_{d-1}\})$ and $\conv(\{v_1,\ldots,v_{d-2},v_{d}\})$ of $C_d$, say $x$ and $y$ respectively. On the other hand, since $C_{d-1}$ has a point in each edge of the triangle $\tau_{d-1}=\conv(v_{d-1},v_d,v_{d+1})$, it has a point $z$ in the interior of the segment $v_{d-1}v_{d}$. By definition, $C_{d-1}$ contains the $(d-2)$-dimensional simplex $C:=\conv(\{v_1,\ldots,v_{d-2},z\})$. By continuity of the barycentric coordinates, it is easy to verify that $C$ must separate $x$ and $y$ within the $(d-1)$-simplex $C_{d}$. Hence, $C_{d-1}\supset C$ must intersect the segment $s=xy$ in its interior, which is a point in the relative interior of $C_d$.

        \item The cases $C_d=\conv(V\setminus \{v_{d}\})$ and $C_d=\conv(V\setminus \{v_{d-1}\})$ are analogous, so we may assume that we are in the former case. By the induction hypothesis we know that there is a point of $\bigcap_{i=1}^{d-2} C_i$ in the relative interior of $\conv(\{v_1,\ldots,v_{d-2},v_{d-1}\})$, say $x$. Therefore $\bigcap_{i=1}^{d-2} C_i$ also contains the segment $xv_{d+1}$, which is contained in the face $C_d$. We claim that $C_{d-1}$ must intersect the interior of this segment. Indeed, let $y$ be the point of $C_{d-1}$ in the segment $v_{d-1}v_{d+1}$. Then $C_{d-1}$ contains  $C=\conv(\{v_1,\ldots,v_{d-2},y\})$. This set must intersect the segment $s=xv_{d+1}$ as desired (for it separates $x$ from $v_{d+1}$ within $C_d$).
        
        \item Finally, assume that $C_d=\conv(V\setminus \{v_{i}\})$ for some $1\le i\le d-2$. As in the previous case, we can find a point of the intersection $\bigcap_{j=1}^{d-2} C_j$ in the interior of the face $\conv(\{v_1,\ldots,v_{d}\}\setminus \{v_i\})$, say $x$. Now we select points $y,z$ of $C_{d-1}$ in the interiors of the segments $v_{d-1}v_{d+1}$ and $v_{d}v_{d+1}$ respectively. The set $\conv(\{y,z,v_1,\dots,v_{d-2}\}\setminus\{v_i\})$ is contained in $C_{d-1}$ and separates $x$ and $v_{d+1}$. Therefore $C_{d-1}$ intersects $s=xv_{d+1}$ in its relative interior as before.\qedhere
    \end{enumerate}
\end{proof}

We are almost done with the construction. In view of Lemma \ref{lem:relint}, we may shrink all the elements of $\hat{\F}^{(d)}$ away from the $(d-2)$-dimensional faces of $\Delta^{(d)}$ in such a way that they remain convex and the colorful intersections continue to be non-empty. In this way we obtain the families $\F_1^{(d)},\dots,\F_{d-1}^{(d)}$. To construct the last family $\F_d^{(d)}$, we take an additional step: we take $m$ parallel copies of each so that they still intersect every element of $\F_1\cup\dots\cup\F_{d-1}$ but are pairwise disjoint.

By the cut-off procedure, no three sets of the same $\F_i$ intersect for $i\in[d-1]$ (as any such intersection would project to a triple intersection within $\T_i$). Thus, in order to pierce any such $\F_i$ at least $\frac{m}{2}=f$ points are needed. To cross $\F=\F_1\cup\dots\cup\F_d$ by lines we also need to cross the relative interiors of the facets of $\Delta$. No line can pierce more than two such interiors. Therefore, at least $\lceil\frac{d+1}{2}\rceil$ lines are needed. This concludes the proof of Theorem \ref{Thm:LowerBound}.

\section{Discussion}\label{Sec:Discus}
We studied families of convex sets which satisfy the Colorful Helly hypothesis. 
Our Theorems \ref{theorem:main} and \ref{theorem:structure} offer complementary relations between the ``transversal dimensions'' of individual color classes.

We conjecture that an even stronger phenomenon happens:

\begin{conjecture}\label{Conjecture}
    For all $1\leq k\leq d$ there exist numbers $h(k,d)$ with the following property.
    For any $d$-colored family $\F$ of convex sets in $\R^d$ with $\CH(\F_1,\dots,\F_d)$ there exist numbers $k_1,\dots,k_d$ so that
    \begin{enumerate}
        \item $\sum_{1\leq i\leq d}k_i\leq d$, and
        \item each color class $\F_i$, for $i\in [d]$, can be crossed by $h(k_i,d)$ $k_i$-flats. 
    \end{enumerate}
\end{conjecture}

It is easy to check that Conjecture \ref{Conjecture} is sharp for families of flats.
The most elementary instance of the conjecture arises for $d=3$ and $\F_3=\{\reals^3\}$. The remaining two classes $\F_1$ and $\F_2$ satisfy a 2-colored hypothesis.
If one of the classes has a transversal by few points, then Conjecture \ref{Conjecture} holds for the families, as the other class can simply be pierced by $\mathbb{R}^3$. Otherwise,
by Lemma \ref{proposition:twocolors} both $\F_1$ and $\F_2$ can be pierced by few planes. Then the validity of Conjecture \ref{Conjecture} in this case depends on the answer to the following question:
\begin{problem}
    Is it true that for any two families $\A,\B$ of convex sets in $\reals^3$ so that $A\cap B\neq \emptyset$ holds for all $A\in \A$ and $B\in \B$, one of the families $\A$ or $\B$ can be crossed by $O(1)$ lines?
\end{problem}

Another intriguing question is what are the ``true'' values of $f'(d)$ and $g'(d)$ for Theorem \ref{theorem:main-d} or, more precisely, what is the relation between these parameters? For example, does the theorem still hold with $f'(d)=1$ and large enough $g'(d)$, as it happens for $d=2$?

\paragraph{Acknowledgements.} The authors thank the anonymous SoCG referees for valuable comments which helped to improve the presentation.

\bibliography{colorbib}
\bibliographystyle{amsplain}

\end{document}